\documentclass[12pt,english]{amsart}
\usepackage[T1]{fontenc}
\usepackage[latin9]{inputenc}
\usepackage{geometry}
\usepackage{enumitem}
\usepackage{amstext}
\usepackage{amsthm}
\usepackage{amssymb}
\usepackage{setspace}
\makeatletter
\numberwithin{figure}{section}
\theoremstyle{plain}
\newtheorem{thm}{\protect\theoremname}[section]
\theoremstyle{remark}
\newtheorem*{acknowledgement*}{\protect\acknowledgementname}
\theoremstyle{plain}
\newtheorem{fact}[thm]{\protect\factname}
\theoremstyle{plain}
\newtheorem{lem}[thm]{\protect\lemmaname}
\theoremstyle{plain}
\newtheorem{prop}[thm]{\protect\propositionname}
\theoremstyle{remark}
\newtheorem{rem}[thm]{\protect\remarkname}
\theoremstyle{plain}

\theoremstyle{plain}
\newtheorem{defn}[thm]{\protect\defnname}

\date{}
\pdfpageattr{/Group << /S /Transparency /I true /CS /DeviceRGB>>}
\usepackage{ae,aecompl}
\usepackage{bbm}
\usepackage{hyperref}
\hypersetup{
	hidelinks = true
}

\let\originalleft\left
\let\originalright\right
\renewcommand{\left}{\mathopen{}\mathclose\bgroup\originalleft}
\renewcommand{\right}{\aftergroup\egroup\originalright}
\usepackage{multicol}

\makeatother

\usepackage{babel}
\addto\captionsbritish{\renewcommand{\acknowledgementname}{Acknowledgement}}
\addto\captionsbritish{\renewcommand{\factname}{Fact}}
\addto\captionsbritish{\renewcommand{\lemmaname}{Lemma}}
\addto\captionsbritish{\renewcommand{\propositionname}{Proposition}}
\addto\captionsbritish{\renewcommand{\questionname}{Question}}
\addto\captionsbritish{\renewcommand{\remarkname}{Remark}}
\addto\captionsbritish{\renewcommand{\theoremname}{Theorem}}
\addto\captionsenglish{\renewcommand{\acknowledgementname}{Acknowledgement}}
\addto\captionsenglish{\renewcommand{\factname}{Fact}}
\addto\captionsenglish{\renewcommand{\lemmaname}{Lemma}}
\addto\captionsenglish{\renewcommand{\propositionname}{Proposition}}
\addto\captionsenglish{\renewcommand{\questionname}{Question}}
\addto\captionsenglish{\renewcommand{\remarkname}{Remark}}
\addto\captionsenglish{\renewcommand{\theoremname}{Theorem}}
\providecommand{\acknowledgementname}{Acknowledgement}
\providecommand{\factname}{Fact}
\providecommand{\lemmaname}{Lemma}
\providecommand{\propositionname}{Proposition}
\providecommand{\questionname}{Question}
\providecommand{\remarkname}{Remark}
\providecommand{\theoremname}{Theorem}
\providecommand{\defnname}{Definition}

\begin{document}
\global\long\def\RR{\mathbb{R}}%

\global\long\def\CC{\mathbb{C}}%

\global\long\def\ZZ{\mathbb{Z}}%

\global\long\def\NN{\mathbb{N}}%

\global\long\def\QQ{\mathbb{Q}}%

\global\long\def\TT{\mathbb{T}}%

\global\long\def\FF{\mathbb{F}}%

\global\long\def\vphi{\varphi}%

\global\long\def\sub{\subseteq}%

\global\long\def\one{\mathbbm1}%

\global\long\def\vol#1{\text{Vol\ensuremath{\left(#1\right)}}}%

\global\long\def\epi#1{\text{epi}\left\{  #1\right\}  }%

\global\long\def\sp{{\rm sp}}%

\global\long\def\K{\mathcal{K}}%

\global\long\def\A{\mathcal{A}}%

\global\long\def\L{\mathcal{L}}%
\global\long\def\P{\mathcal{P}}%

\global\long\def\W{\mathcal{W}}%

\global\long\def\iprod#1#2{\langle#1,\,#2\rangle}%

\global\long\def\conv{{\rm Conv}}%

\global\long\def\eps{\varepsilon}%

\global\long\def\norm#1{\left\Vert #1\right\Vert }%

\global\long\def\supp#1{{\rm supp}\left(#1\right)}%

\global\long\def\flr#1{\left\lfloor #1\right\rfloor }%

\global\long\def\ceil#1{\left\lceil #1\right\rceil }%

\global\long\def\hseg#1{\left\llbracket #1\right\rrbracket }%

\global\long\def\EE{\mathbb{E}}%
$ $

\global\long\def\RR{\mathbb{R}}%

\global\long\def\ZZ{\mathbb{Z}}%

\global\long\def\ll{\preceq}%

\global\long\def\bp#1{\big(#1\big)}%

\global\long\def\Bp#1{\Big(#1\Big)}%

\title{Discrete variants of Brunn-Minkowski type inequalities}
\author{Diana Halikias}
\address{Department of Mathematics, Yale University, New Haven, CT 06511, USA}
\email{diana.halikias@yale.edu}
\author{Bo'az Klartag}
\address{Department of Mathematics, Weizmann Institute of Science, Rehovot
76100 Israel.}
\email{boaz.klartag@weizmann.ac.il}
\author{Boaz A. Slomka}
\address{Department of Mathematics, the Open University of Israel, Ra'anana
4353701 Israel}
\email{slomka@openu.ac.il}
\begin{abstract}
We present an alternative, short proof of a recent discrete version
of the Brunn-Minkowski inequality due to Lehec and the second named author. Our proof also yields the four functions theorem of Ahlswede and Daykin as well as some new variants.
\end{abstract}

\maketitle

\section{Introduction}
Correlation inequalities such as the Fortuin-Kasteleyin-Ginibre (FKG) inequality are of use in
the analysis of several models in probability theory and statistical
physics (see, e.g., Grimmett \cite{Grimmett_Perc, Grimmett_Prob}). These inequalities are closely related to the following four functions theorem of Ahlswede and Daykin \cite{Ahlswede1978}:
\begin{thm}
\label{thm:4fcns}Suppose that $f,g,h,k:\ZZ^{n}\to[0,\infty)$ satisfy
\[
f\left(x\right)g\left(y\right)\le h\left(x\wedge y\right)k\left(x\vee y\right) \qquad  \forall x=\left(x_{1},\dots,x_{n}\right),y=\left(y_{1},\dots,y_{n}\right)\in\ZZ^{n}
\]
where $x\wedge y=\left(\min\left(x_{1},y_{1}\right),\dots,\min\left(x_{n},y_{n}\right)\right)$,
and $x\vee y=\left(\max\left(x_{1},y_{1}\right),\dots,\max\left(x_{n},y_{n}\right)\right)$.
Then
\[
\Bp{\sum_{x\in\ZZ^{n}}f\left(x\right)}\Bp{\sum_{x\in\ZZ^{n}}g\left(x\right)}\le\Bp{\sum_{x\in\ZZ^{n}}h\left(x\right)}\Bp{\sum_{x\in\ZZ^{n}}k\left(x\right)}.
\]
\end{thm}

Theorem \ref{thm:4fcns} is usually formulated under the additional
assumption that $f,g,h,k$ are all supported in the discrete cube $\{0,1\}^{n}$.
It was suggested by Gozlan, Roberto, Samson, and Tetali \cite{Gozlan2019}
that Theorem \ref{thm:4fcns} is connected with a discrete variant of
the Brunn-Minkowski inequality, recently proven by Lehec and the second
named author \cite[Theorem 1.4]{Klartag2019}, which is the case $\lambda=1/2,n=1$ of the following theorem:
\begin{thm}
\label{thm:lambda_BM}Let $\lambda\in\left[0,1\right]$ and suppose that
$f,g,h,k:\ZZ^{n}\to[0,\infty)$ satisfy
\[
f\left(x\right)g\left(y\right)\le h\left(\left\lfloor \lambda x+\left(1-\lambda\right)y\right\rfloor \right)k\left(\ceil{\left(1-\lambda\right)x+\lambda y}\right) \qquad \forall x,y\in\ZZ^{n}
\]
where $\flr x=\left(\flr{x_{1}},\dots\flr{x_{n}}\right)$ and $\ceil x=\left(\ceil{x_{1}},\dots,\ceil{x_{n}}\right)$.
Then
\[
\Bp{\sum_{x\in\ZZ^{n}}f\left(x\right)}\Bp{\sum_{x\in\ZZ^{n}}g\left(x\right)}\le\Bp{\sum_{x\in\ZZ^{n}}h\left(x\right)}\Bp{\sum_{x\in\ZZ^{n}}k\left(x\right)}.
\]
\end{thm}

\noindent Here $\lfloor r\rfloor = \max \{ m \in \ZZ \, ; \, m \leq r \}$ is the lower integer part of $r\in\RR$
and $\lceil r\rceil=-\lfloor-r\rfloor$ the upper integer part. A
standard limiting argument (see \cite[Section 2.3]{Gozlan2019})
leads from the case $\lambda = 1/2, h=k$ of Theorem \ref{thm:lambda_BM}
to the case $\lambda = 1/2$ of the Brunn-Minkowski inequality in its multiplicative form:
\begin{equation*}  \vol{ \frac{A+B}{2} } \geq \sqrt{ \vol{A} \vol{B}},  \end{equation*}
where $A + B = \{ x + y \, ; \, x \in A, y \in B \}$, where $A, B \subseteq \RR^n$ are any Borel-measurable sets, and $\vol{\cdot}$ stands for the $n$-dimensional Lebesgue volume.
The proof in \cite{Klartag2019} for the case $n=1,\lambda=1/2$, which involves stochastic analysis, admits
a straightforward generalization to the more general case described above. An alternative argument using ideas from the theory of optimal transport was given by Gozlan, Roberto, Samson and Tetali~\cite{Gozlan2019}.

\medskip
Our goal in this note is to provide a unified proof of Theorem \ref{thm:4fcns}
and Theorem \ref{thm:lambda_BM}, which is perhaps  as
elementary as the original proof of the four functions theorem by Ahlswede
and Daykin \cite{Ahlswede1978}. The first issue that we would
like to address, is the identification of the relevant common features of operations such as
\begin{equation}
x\wedge y,x\vee y,\left\lfloor \frac{x+y}{2}\right\rfloor ,\left\lceil \frac{x+y}{2}\right\rceil ,\left\lfloor \lambda x+(1-\lambda)y\right\rfloor ,\left\lceil \lambda x+(1-\lambda)y\right\rceil ,\ldots\label{eq_1456}
\end{equation}
that are defined for $x,y\in\ZZ^{n}$, with $0<\lambda<1$.

\medskip
Our observation is that these operations $T:\ZZ^{n}\times\ZZ^{n}\rightarrow\ZZ^{n}$
satisfy two axioms:
\begin{enumerate}[labelindent=\parindent,leftmargin={*},label=(P\arabic*),align=left]
\item \label{enu:Additivity}$T$ is {\em translation equivariant}: $T\left(x+z,y+z\right)=T\left(x,y\right)+z$
for all $z\in\ZZ^n$.
\item \label{enu:Monotonicity} $T$ is {\em monotone in the sense of Knothe}  with respect to some decomposition of $\ZZ^n$ into a direct sum of groups $\ZZ^n = G_1 \times \dots \times G_k$. That is, $T = (T_1,\ldots,T_k)$ where for each $i\in\{1,\dots,k\}$:

    \begin{enumerate}
    \item[(i)] $T_i:(G_1\times\dots\times G_i)\times(G_1\times\dots\times G_i)\to G_i$. In other words, $T_i(x,y)$ depends only on the first $i$ coordinates of its  arguments $x,y \in G_1 \times \dots \times G_k$.
        \item[(ii)] There exists a total additive ordering $\ll_i$ on $G_i$ such that $T^{(a,b)}_i:G_i\times G_i \to G_i$ defined  by $T_i^{(a,b)}(x,y)=T_i\big((a,x),(b,y)\big)$ for $a,b\in G_1\times\dots\times G_{i-1}$ satisfies
            $$ x_1\ll_i x_2, \ y_1 \ll_i y_2 \qquad \Longrightarrow \qquad  T_i^{(a,b)}(x_1,y_1)\ll_i T_i^{(a,b)}(x_2,y_2) $$
            for all $a,b \in G_1\times\dots\times G_{i-1}$ and $x_1,x_2, y_1, y_2 \in G_i$.
\end{enumerate}
\end{enumerate}

Maps that satisfy a condition similar to (P2)  were used by Knothe  \cite{knothe} in his proof of the Brunn-Minkowski inequality. 
In the language of stochastic processes, one could say that the map $T$ is adapted to the filtration induced by the decomposition $\ZZ^n = G_1 \times ... \times G_k$, or that the map $T$ cannot see into the future and 
 it is monotone when conditioning on the past.
Recall that a total ordering $\preceq$ on an abelian group $G$ is a binary relation which is reflexive, anti-symmetric and transitive, such that for any distinct $x,y$, either $x\preceq y$ or else $y\preceq x$. An ordering $\preceq$ is additive if for all $x,y,z$,
\[
x\ll y\implies x+z\ll y+z.
\]
We remark that the requirement of existence of a total additive ordering on a finitely-generated abelian group $G$, forces $G$ to be isomorphic to $\ZZ^{\ell}$
for some $\ell$.

\medskip
The standard cartesian decomposition of $\ZZ^n = \ZZ \times \dots \times \ZZ$ into one-dimensional groups, each of which equipped with the standard order on $\ZZ$, attests to the fact that all the examples in  \eqref{eq_1456} satisfy properties \ref{enu:Additivity} and \ref{enu:Monotonicity}. In these examples, each  $T_i$ is a function from $G_i\times G_i$ to $G_i$.

\medskip Another natural example for an additive, total ordering on  $\ZZ^{n}$ (or on $\RR^n$) is the standard lexicographic order relation.
 Given an additive, total ordering $\preceq$ on $\RR^n$ and an invertible, linear map $L: \RR^n \rightarrow \RR^n$ we may construct another additive, total ordering $\preceq_L$ by requiring that $x \preceq_L y$ if and only if $Lx \preceq Ly$.
For an additive, total ordering $\preceq$ on $\ZZ^{n}$ the operations $\max(x,y)$ and $\min(x,y)$ are well-defined, and they satisfy properties \ref{enu:Additivity} and \ref{enu:Monotonicity} with $k = 1$.

\medskip
Yet another example for an operation $T:\ZZ^n\times \ZZ^n\to\ZZ^n$ that satisfies properties  \ref{enu:Additivity} and \ref{enu:Monotonicity} is given by $T=(T_1,\ldots,T_n)$ where
\begin{equation}\label{eq:Ollivier-Villani}
T_i(x,y)=\begin{cases}
x_{i}, & \#\left\{ j\le i\,\,;\,\,x_{j}\neq y_{j}\right\} \text{ is odd}\\
y_{i}, & \text{otherwise}
\end{cases}.
\end{equation}
We prove the following:
\begin{thm}
\label{thm:main}Let $T:\ZZ^{n}\times\ZZ^{n}\rightarrow\ZZ^{n}$ satisfy
properties \ref{enu:Additivity} and \ref{enu:Monotonicity}. Suppose that
$f,g,h,k:\ZZ^{n}\to[0,\infty)$ satisfy
\[
f\left(x\right)g\left(y\right)\le h\left(T\left(x,y\right)\right)k\left(x+y-T\left(x,y\right)\right) \qquad \qquad \forall x,y\in\ZZ^{n}.
\]
Then
\[
\Bp{\sum_{x\in\ZZ^{n}}f\left(x\right)}\Bp{\sum_{x\in\ZZ^{n}}g\left(x\right)}\le\Bp{\sum_{x\in\ZZ^{n}}h\left(x\right)}\Bp{\sum_{x\in\ZZ^{n}}k\left(x\right)}.
\]
\end{thm}

Clearly Theorem \ref{thm:4fcns} and Theorem \ref{thm:lambda_BM} follow from Theorem \ref{thm:main}. See also Borell \cite[Theorem 2.1]{borell} for Brunn-Minkowski type inequalities for operations other than Minkowski sum with monotonicity properties.

\medskip One can relax the monotonicity property \ref{enu:Monotonicity} by replacing it with another ``exclusion'' property, which requires no ordering at all. We formulate this property, as well as our next theorem, in greater generality, with $\ZZ^{n}$ replaced by a finitely generated abelian group $G$, and $T:\ZZ^{n}\times\ZZ^{n}\to\ZZ^{n}$ replaced by $T:G\times G\to G$. It is well-known that any such $G$ is isomorphic to $\ZZ^{n}\times(\ZZ/p_{1}\ZZ)\times\dots\times(\ZZ/p_{k} \ZZ)$ for some powers of primes $p_{1},\dots,p_{k}$.

\begin{defn}\label{def:exlucsion}
{\rm We say that an operation $T:G \times G\to G$ is} {\it exclusive} {\rm if for every finite set $A\sub G$ with at least two elements, and all $z\in G$, there exist distinct $x,y\in A$ such that for $A_{1}=A\setminus\left\{ x\right\} $, $A_{2}=A\setminus\left\{ y\right\} $, and $A_{3}=A\setminus\left\{ x,y\right\}$, the following conditions holds:
	\begin{enumerate}[label=(\alph*)]
		\item \label{enu:T123_a}
		$\{T\left(x,z-y\right),T\left(y,z-x\right)\}\not\sub T\left(A_{i},z-A_{i}\right)$ for $i\in\{1,2\}$,
		\item \label{enu:T123_b}
		$\{T\left(x,z-y\right),T\left(y,z-x\right)\}\cap T\left(A_{3},z-A_{3}\right)=\emptyset,$\smallskip

\noindent where $T(A_i, z-A_i) = \{ T(u,z-v) \, ; \, u,v \in A_i\}$.
\end{enumerate}}
\end{defn}

\noindent In the next theorem, we replace \ref{enu:Monotonicity} by the following property:
\begin{enumerate}[labelindent=\parindent,leftmargin={*},label=(P\arabic*'), start=2,align=left]
	\item \label{enu:T123} There exists a decomposition of abelian groups $G = G_1\times \dots \times G_k$ such that
	\begin{enumerate}
	
		\item[(i)]$T=(T_1,\dots,T_k)$ with $T_i:(G_1\times\dots\times G_i)\times(G_1\times\dots\times G_i)\to G_i$ for each $i$.
		\item[(ii)] For all $i\in\{1,\dots,k\}$ and $a,b\in G_1\times\dots\times G_{i-1}$ the operation $T_i^{(a,b)}:G_i\times G_i\to G_i$ defined by $T^{(a,b)}_i(x,y)=T_i((a,x) ,(b,y))$ is exclusive.
	\end{enumerate}
\end{enumerate}

We prove the following:

\begin{thm}
\label{thm:main_gen}Let $\left(G,+\right)$ be a finitely generated abelian group, and $T:G\times G\to G$ satisfy   \ref{enu:Additivity} and \ref{enu:T123}. Suppose that $f,g,h,k:G\to[0,\infty)$ satisfy
\[
f\left(x\right)g\left(y\right)\le h\left(T\left(x,y\right)\right)k\left(x+y-T\left(x,y\right)\right) \qquad \forall x,y\in G.
\]
Then
\[
\Bp{\sum_{x\in G}f\left(x\right)}\Bp{\sum_{x\in G}g\left(x\right)}\le\Bp{\sum_{x\in G}h\left(x\right)}\Bp{\sum_{x\in G}k\left(x\right)}.
\]
\end{thm}

The next two sections are devoted to the proofs of the above theorems.
We additionally include a final section with commentary on the applicability of this work to related inequalities,
such as the ones proven by Cordero-Erausquin and Maurey \cite{Cordero-Erausquin2017},
 Iglesias, Yepes Nicol\'as and Zvavitch \cite{INZ}
and Ollivier and Villani \cite{Ollivier-Villani}

\begin{acknowledgement*}
The first named author would like to express her gratitude to the Kupcinet-Getz International Summer School at the Weizmann Institute for supporting this research. The second named author was partially supported by a grant from the Israel Science Foundation (ISF). We thank the anonymous referee for the useful comments and suggestions.
\end{acknowledgement*}

\section{\label{sec:proof_main_gen}Proof of Theorem \ref{thm:main_gen}}

The core of this paper is the proof of Theorem \ref{thm:main_gen} in the particular case where $T$ itself is exclusive, which is Proposition \ref{prop:core} below. We begin with the following elementary fact:
\begin{fact}
\label{Fact:max_sum}Suppose $a,b,c,d\ge0$. If $ab\le cd$ and $\max\left\{ a,b\right\} \le\max\left\{ c,d\right\} $
then $a+b\le c+d$.
\end{fact}

\begin{proof}
\noindent Pick $A \geq a, B \geq b$ such that $\max \{A, B \} \leq \max\left\{ c,d\right\} $ and $AB = cd = r$.  Then $(A-B)^2 \leq (c-d)^2$ and so
$ (a+b)^2 \leq (A+B)^2 = 4 r + (A - B)^2 \leq 4r + (c-d)^2 = (c+d)^2$.
\end{proof}
Recall that under the assumptions of Theorem \ref{thm:main_gen} we
have $f,g,h,k:G\to[0,\infty)$ satisfying
\begin{equation}
f\left(x\right)g\left(y\right)\le h\left(T\left(x,y\right)\right)k\left(x+y-T\left(x,y\right)\right) \qquad \forall x,y\in G.\label{eq:ineq_cond}
\end{equation}

\noindent For $j,z\in G$ denote $F_{z}\left(j\right)=f\left(j\right)g\left(z-j\right)$
and $H_{z}\left(j\right)=h\left(j\right)k\left(z-j\right)$. Note
that, by (\ref{eq:ineq_cond}),
\begin{equation}\label{eq:55}
F_{z}\left(j\right)\le H_{z}\left(T\left(j,z-j\right)\right).
\end{equation}
We claim that for all $i,j,z\in G$ we have
\begin{equation}
F_{z}\left(i\right)F_{z}\left(j\right)\le H_{z}\left(T\left(i,z-j\right)\right)H_{z}\left(T\left(j,z-i\right)\right).\label{claim:FiFj}
\end{equation}

\noindent Indeed, by (\ref{eq:ineq_cond}) and \ref{enu:Additivity}
we have
\begin{align*}
F_{z}\left(i\right)F_{z}\left(j\right) & = \, f\left(i\right)g\left(z-i\right)f\left(j\right)g\left(z-j\right)=f\left(i\right)g\left(z-j\right)f\left(j\right)g\left(z-i\right)\\
 & \le \, h\left(T\left(i,z-j\right)\right)k\left(z+\left(i-j\right)-T\left(i,z-j\right)\right)h\left(T\left(j,z-i\right)\right)k\left(z+\left(j-i\right)-T\left(j,z-i\right)\right)\\
 & =h\left(T\left(i,z-j\right)\right)k\left(z-T\left(i,z-j\right)\right)h\left(T\left(j,z-i\right)\right)k\left(z-T\left(j,z-i\right)\right)\\
 & = \, H_{z}\left(T\left(i,z-j\right)\right)H_{z}\left(T\left(j,z-i\right)\right).
\end{align*}

\begin{prop}\label{prop:core}
Let $\left(G,+\right)$ be a finitely generated abelian group, and let $T:G\times G\to G$ be an exclusive operation that satisfies \ref{enu:Additivity}. Suppose that $f,g,h,k:G\to[0,\infty)$ satisfy
		\[
		f\left(x\right)g\left(y\right)\le h\left(T\left(x,y\right)\right)k\left(x+y-T\left(x,y\right)\right) \qquad \forall x,y\in G.
		\]
		Then
		\[
		\Bp{\sum_{x\in G}f\left(x\right)}\Bp{\sum_{x\in G}g\left(x\right)}\le\Bp{\sum_{x\in G}h\left(x\right)}\Bp{\sum_{x\in G}k\left(x\right)}.
		\]
	\end{prop}

\begin{proof}We need to prove that
$$ \sum_{j, z \in G} F_z(j) \leq \sum_{j, z \in G} H_z(j). $$
Fix $z\in G$. It is sufficient to prove that for every finite set $A\sub G$,
\begin{equation} \label{eq_1139}
\sum_{j\in A}F_{z}\left(j\right)\le\sum_{j\in T\left(A,z-A\right)}H_{z}\left(j\right).
\end{equation}

\noindent We proceed to prove so by induction on $n=\left|A\right|$.
\\
\textbf{Induction base}: For $n=0$ the statement is vacuous, as the empty sum equals zero. For $n=1$ the statement holds by \eqref{eq:55}.

\noindent \textbf{Induction step}: Assume $n \geq 2$ and that the statement holds for
all $m\le n-1$. Let $A\sub G$ with $\left|A\right|=n$. By assumption,
there exist distinct $x,y\in A$ such that assertions \ref{enu:T123_a} and \ref{enu:T123_b} of Definition \ref{def:exlucsion}
are satisfied. By switching $x$ with $y$ if necessary,
we may assume that $F_{z}\left(x\right)\le F_{z}\left(y\right)$.
By (\ref{claim:FiFj}) we have
\begin{equation} \label{eq_1141}
F_{z}\left(x\right)F_{z}\left(y\right)\le H_{z}\left(T\left(x,z-y\right)\right)H_{z}\left(T\left(y,z-x\right)\right).
\end{equation}

\noindent \underline{Case 1}: Assume $F_{z}\left(y\right)\ge\max\left\{ H_{z}\left(T\left(x,z-y\right)\right),H_{z}\left(T\left(y,z-x\right)\right)\right\} $. Then, it follows from (\ref{eq_1141}) that
\begin{equation} F_{z}\left(x\right)\le\min\left\{ H_{z}\left(T\left(x,z-y\right)\right),H_{z}\left(T\left(y,z-x\right)\right)\right\}. \label{eq_1142} \end{equation}
The induction hypothesis for $A_{1}=A\setminus{\{x\}}$ tells us that
\begin{equation} \label{eq_1143}
\sum_{j\in A_{1}}F_{z}\left(j\right)\le\sum_{j\in T\left(A_{1},z-A_{1}\right)}H_{z}\left(j\right).
\end{equation}
By adding inequalities (\ref{eq_1142}) and (\ref{eq_1143}), and using property \ref{enu:T123_a} of Definition \ref{def:exlucsion},
we obtain the desired inequality (\ref{eq_1139}).

\noindent \underline{Case 2}: Assume $F_{z}\left(y\right)\le\max\left\{ H_{z}(T\left(x,z-y\right)),H_z(T\left(y,z-x\right))\right\} $.
Since $F_z(x)\le F_z(y)$, we may apply (\ref{eq_1141}) and Fact \ref{Fact:max_sum} and obtain
\begin{equation}\label{eq:56}
F_{z}(x) + F_z(y) \leq  H_{z}\left(T\left(x,z-y\right)\right)+H_{z}\left(T\left(y,z-x\right)\right).
\end{equation}
Note that  $T(x,z-y)\neq T(y,z-x)$ as $T(y,z-x)-T(x,z-y)=y-x \neq 0$. Therefore, by combining \eqref{eq:56} with the induction hypothesis for $A_{3}=A\setminus{\{x,y\}}$ and property (b),
we deduce the inequality (\ref{eq_1139}). This completes our proof.
\end{proof}

\begin{proof}[Proof of Theorem \ref{thm:main_gen}]
	
	We proceed by induction on $k$, the number of groups participating in the decomposition of $G$  given in \ref{enu:T123}. For $k=1$, the statement is equivalent to that in Proposition \ref{prop:core}. Assume next that  $k\ge2$ and that the statement holds true for  $k-1$.

\noindent Denote $G'=G_2\times\dots\times G_{k}$. For $a,b\in G_1$ and $x',y'\in G'$  denote
$$f^a(x')=f(a,x'), \quad g^a(x')=g(a,x'), \quad h^{a}(x')=h(a,x'), \quad k^{a}(x')=k(a,x').$$



Fix $a, b\in G_1$. For  $i\in\{2,\dots,k\}$, define $T'_i:(G_2\times\dots\times G_i)\times (G_2\times\dots\times G_i)\to G_i$ by $T'_i(x,y)=T_i\big((a,x),(b,y)\big)$, and  $T':G'\times G'\to G'$ by $T'=(T'_2, \dots, T'_k)$. Note that $T'$ satisfies \ref{enu:Additivity} and \ref{enu:T123} with respect to this decomposition.
The assumptions of the theorem tell us that
$$
f^a(x')g^b(y')\le h^{T_1(a,b)}\big(T'(x',y')\big)k^{a+b-{T}_1(a,b)}\big(x'+y'-T'(x',y')\big)\quad\quad\forall x',y'\in G'.
$$
By the induction hypothesis, it follows that
$$
\left( \sum_{x'\in G'}f^a(x') \right) \left( \sum_{x'\in G'}g^b(x') \right) \le \left( \sum_{x'\in G'}h^{T_1(a,b)}(x') \right) \left( \sum_{x'\in G'}k^{a+b-{T}_1(a,b)}(x') \right).
$$

\noindent For every $a\in G_1$ set
$$F(a)=\sum_{x'\in G'} f^a(x'), \ \ G(a)=\sum_{x'\in G'} g^a(x'), \ \ H(a)=\sum_{x'\in G'} h^a(x') \ \ \text{ and } \ \ K(a)=\sum_{x'\in G'} k^a(x').$$

\noindent Rewriting the previous inequality, we have  for all $a,b \in G_1$,
$$F(a)G(b)\le H(T_1(a,b))K(a+b-T_1(a,b)). $$
Since $T_1: G_1 \times G_1 \rightarrow G_1$ is an exclusive map satisfying (P1), we may apply Proposition \ref{prop:core}  and conclude  that
$$ \left( \sum_{a\in G_1} F(a) \right) \left( \sum_{a\in G_1} G(a) \right) \le \left( \sum_{a\in G_1} H(a) \right) \left( \sum_ {a\in G_1}K(a) \right).$$
This completes the proof.
\end{proof}

\section{\label{sec:Proof_Main}Proof of Theorem \ref{thm:main}}

Theorem \ref{thm:main} is an  immediate consequence of Theorem \ref{thm:main_gen} due to the following observation:
\begin{lem}
\label{claim:Mon_Implies_12}Suppose $T:\ZZ^{n}\times\ZZ^{n}\to\ZZ^{n}$
satisfies properties \ref{enu:Additivity} and \ref{enu:Monotonicity}.
Then $T$ satisfies property \ref{enu:T123}.
\end{lem}

\begin{proof}
	
We first show that $T_1$ is exclusive. To this end, let $m\ge 2$ and $z\in G_1$. Suppose $A=\left\{ x_{1},\dots,x_{m}\right\}\sub G_1$, where $x_{1}\prec_1\dots\prec_1 x_{m}$. Here $a\prec_1 b$
	means that $a\ll_1 b$ and $a\neq b$. Set $x=x_{1},y=x_{m}$, and recall that $A_1=A\setminus{\{x\}}$, $A_2=A\setminus{\{y\}}$, and $A_3=A\setminus{\{x,y\}}$. For a finite subset $S \subseteq G_1$ we write $\max S$ and $\min S$ for the maximal
and minimal elements with respect to the total order $\preceq_1$. By properties (P1) and (P2) we have
\begin{align*}
		\max_{v, w \in A_1} T_1\left(w,z-v\right) & -\min_{v, w \in A_1}  T_1\left(v,z-w\right) = T_1\left(x_{m},z-x_{2}\right)-T_1\left(x_{2},z-x_{m}\right) \\ & =x_{m}-x_{2}\prec_1 x_{m}-x_{1}=T_1\left(x_{m},z-x_{1}\right)-T_1\left(x_{1},z-x_{m}\right).
	\end{align*}
	Therefore,  $\{T_1 \left(x_{1},z-x_{m}\right), T_1 \left(x_{m},z-x_{1}\right)\}\not\sub T_1 \left(A_{1},z-A_{1}\right)$.
	A similar argument shows that the same holds for $A_{2}$. This verifies condition \ref{enu:T123_a} of Definition \ref{def:exlucsion}.
To verify condition \ref{enu:T123_b} of Definition \ref{def:exlucsion}, note that if $m=2$ then $A_3=\emptyset$, and hence condition \ref{enu:T123_b} holds trivially. Otherwise, letting  $v=\min\left\{x_{i+1}-x_{i}\,;\,i\in\left\{ 1,\dots,m-1\right\} \right\} \succ0$, we have
\begin{align*}
		\max_{v,w\in A_{3}}T_1 \left(v,z-w\right) & =T_1 \left(x_{m-1},z-x_{2}\right)\\
		& \ll T_1 \left(x_{m}-v,z-x_{1}-v\right)=T_1 \left(x_{m},z-x_{1}\right)-v\prec T_1 \left(x_{m},z-x_{1}\right).
	\end{align*}
	Therefore,  $T_1 \left(x_{m},z-x_{1}\right)\not\in T_1 \left(A_{3},z-A_{3}\right)$. Similarly, $T_1 \left(x_{1},z-x_{m}\right)\not\in T_1 \left(A_{3},z-A_{3}\right)$, which verifies condition \ref{enu:T123_b} of Definition \ref{def:exlucsion}. Hence $T_1$ is an exclusive map.
We proceed by induction on the number of groups $k$ in the decomposition of $G$  given in property \ref{enu:T123}. For $k=1$, we verified above that the statement holds for  $T=T_1$.

Let $k \geq 2$ and assume that the statement holds for a decomposition into $k-1$ groups. Fix $a_1,b_1\in G_1$ and let $G'=G_2\times\dots\times G_k$.
For $i\in\{2,\dots,k\}$, define $T'_i:(G_2\times\dots\times G_i)\times (G_2\times\dots\times G_i)\to G_i$ by $T'_i(x,y)=T_i\big((a_1,x),(b_1,y)\big)$ and  $T':G'\times G'\to G'$ by $T'=(T'_2,\dots,T'_k)$. Note that $T'$ satisfies \ref{enu:Additivity} and \ref{enu:Monotonicity} with respect to this decomposition. By the induction hypothesis, $T'$ satisfies \ref{enu:T123}.

This means that for all $i\in\{2,\dots,k\}$ and $a=(a_1,a'),b=(b_1,b')\in G_1\times\dots\times G_{i-1}$, the map $T^{(a,b)}_i=T'^{\,(a',b')}_i$ is exclusive.  Since $a_1$ and $b_1$ are arbitrary, it follows that  $T^{(a,b)}_i$ is exclusive for all $i\in\{1,\dots,k\}$ and $a,b\in G_1\times\dots\times G_{i-1}$, and thus $T$ satisfies \ref{enu:T123}.
\end{proof}

\begin{rem}\label{rem:1}
	
	Using Lemma \ref{claim:Mon_Implies_12}, one can show that the operations in \eqref{eq_1456} do not satisfy \ref{enu:Monotonicity} without decomposing $\ZZ^n$ into a direct sum of more than one group.
	To see this, consider e.g., $T(x,y)=x\vee y$, defined for $x,y\in\ZZ^2$. By Lemma \ref{claim:Mon_Implies_12}, it is sufficient to show that $T$ is not exclusive. A direct inspection of the set $A=\{(-1,-1),(-1,1),(1,-1),(1,1)\} \subset \ZZ^2$ and the point $z=(0,0) \in \ZZ^2$ shows that $T$ indeed violates the conditions of Definition \ref{def:exlucsion}.
\end{rem}

\section{Related inequalities}
\subsection{\label{sec:BM} Continuous Brunn-Minkowski type inequalities}

The classical Brunn-Minkowski inequality states that for any two non-empty Borel-measurable subsets of $\RR^n$, one has
\[
\vol{A+B}^{1/n}\ge\vol A^{1/n}+\vol B^{1/n}.
\]
In its equivalent dimension-free form, it states that for any $\lambda\in\left[0,1\right]$,
\[
\vol{\lambda A+\left(1-\lambda\right)B}\ge\vol A^{\lambda}\vol B^{1-\lambda}.
\]

A functional form of the Brunn-Minkowski inequality, known as the
Prékopa-Leindler inequality, states that for any Borel functions $f,g,h:\RR^{n}\to[0,\infty)$
and any $\lambda\in\left[0,1\right]$ such that $f\left(x\right)^{\lambda}g\left(y\right)^{1-\lambda}\le h\left(\lambda x+\left(1-\lambda\right)y\right)$
for all $x,y\in\RR^{n}$, we have
\begin{equation}
\left(\int_{\RR^n} f\left(x\right)\,dx\right)^{\lambda}\left(\int_{\RR^n} g\left(x\right)\,dx\right)^{1-\lambda}\le\int_{\RR^n} h\left(x\right)\,dx.\label{eq:Classical_PL}
\end{equation}
See, e.g., the first pages in Pisier \cite{pis} for proofs of these inequalities.
When $\lambda=1/2$ and $h =k$, the analogy between Theorem \ref{thm:lambda_BM}
and the Pr\'ekopa--Leindler inequality is evident, see \cite[Section 2.3]{Gozlan2019} for a standard limiting argument
that leads from Theorem \ref{thm:lambda_BM} to \eqref{eq:Classical_PL}. For $\lambda\neq1/2$, a similar limiting
argument leads to a weighted variant of the Pr\'ekopa--Leindler inequality due to Cordero-Erausquin and Maurey
\cite{Cordero-Erausquin2017}:
\begin{thm}
\label{thm:Cordero_Maurey}Let $\lambda\in\left[0,1\right]$. Suppose
$f,g,h,k:\RR^{n}\to[0,\infty)$ are measurable functions satisfying
\[
f\left(x\right)g\left(y\right)\le h\left(\lambda x+\left(1-\lambda\right)y\right)k\left(\left(1-\lambda\right)x+\lambda y\right)\,\,\forall x,y\in\RR^{n}.
\]
Then
\[
\left(\int_{\RR^n} f\left(x\right)\,dx\right)\left(\int_{\RR^n} g\left(x\right)\,dx\right)\le\left(\int_{\RR^n} h\left(x\right)\,dx\right)\left(\int_{\RR^n} k\left(x\right)\,dx\right).
\]
\end{thm}

\noindent Note that for $\lambda=1/2$ and $h=k$, Theorem \ref{thm:Cordero_Maurey}
coincides with (\ref{eq:Classical_PL}). We omit the details of the standard limiting argument leading from Theorem \ref{thm:lambda_BM} to
Theorem \ref{thm:Cordero_Maurey}, as they are almost identical to the argument in \cite[Section 2.3]{Gozlan2019}.
Another inequality in the spirit of Theorem \ref{thm:Cordero_Maurey}
is the following limit case of Theorem \ref{thm:4fcns}. Again, the limiting argument is standard
and it is omitted.
\begin{thm}
Suppose $f,g,h,k:\RR^{n}\to[0,\infty)$ are Borel functions satisfying
\begin{equation}
f\left(x\right)g\left(y\right)\le h\left(x\wedge y\right)k\left(x\vee y\right)\,\,\forall x=\left(x_{1},\dots,x_{n}\right),y=\left(y_{1},\dots,y_{n}\right)\in\RR^{n}
\label{eq_412} \end{equation}
where $x\wedge y=\left(\min\left(x_{1},y_{1}\right),\dots,\min\left(x_{n},y_{n}\right)\right)$,
and $x\vee y=\left(\max\left(x_{1},y_{1}\right),\dots,\max\left(x_{n},y_{n}\right)\right)$.
Then
\[
\left(\int_{\RR^n} f\left(x\right)\,dx\right)\left(\int_{\RR^n} g\left(x\right)\,dx\right)\le\left(\int_{\RR^n} h\left(x\right)\,dx\right)\left(\int_{\RR^n} k\left(x\right)\,dx\right).
\] \label{thm_413}
\end{thm}

Another possibility, is to replace the operations $x \wedge y$ and $x \vee y$ in (\ref{eq_412}) by the operations
$\min(x,y)$ and $\max(x,y)$ with respect to the standard lexicographic order on $\RR^n$. The conclusion
of Theorem \ref{thm_413} holds true in this case as well, being a limiting case of Theorem 1.3, as the reader may verify.

\subsection{\label{sec:Ollivier-Villani} A discrete Brunn-Minkowski inequality I}
In Ollivier and Villani \cite{Ollivier-Villani}, a Brunn-Minkowski type inequality with curvature terms is proved on the discrete hypercube. A simplified version of their inequality, without curvature, states that for any sets $A,B \subseteq \{0,1\}^n$ one has
\begin{equation}  \# M\ge \sqrt{\# A \, \cdot \, \#B}, \label{eq_257} \end{equation}
 where $M$ is the set of all midpoints of pairs $(a,b)$ with $a\in A$ and $b \in B$.
In the terminology of \cite{Ollivier-Villani}, a point  $m=(m_1,\ldots,m_n) \in \{0, 1\}^n$ is a midpoint of two points $a, b \in \{0,1\}^n$ if $m_i = a_i$ whenever $a_i = b_i$, and
$$  \# \{ 1 \leq j  \leq n \, ; m_j = a_j \} = \# \{ 1 \leq j \leq n \, ; m_j = b_j \} + \eps $$
with $\eps \in \{-1,0,1 \}$, i.e., essentially half of the remaining bits of $m$ coincide with those of $a$ and the other half with those of $b$.

 We show that inequality (\ref{eq_257}) holds for a much smaller subset of midpoints. For example, let us use the  operation $T$ given in \eqref{eq:Ollivier-Villani}. Recall that $T(a,b)=(T_1(a,b),\ldots,T_n(a,b))$ is defined by:
$$T_i(a,b)=\begin{cases}
	a_{i}, & \#\left\{ j\le i\,\,:\,\,a_{j}\neq b_{j}\right\} \text{ is odd}\\
	b_{i}, & \text{otherwise}
\end{cases}$$
It is clear that $T(a,b)$ is one of the midpoints of $a$ and $b$ in the sense of \cite{Ollivier-Villani}, as well as the point $a+b - T(a,b)$.
Denote $$ M_1^- = \bigcup_{a \in A, b \in B} T(a,b), \qquad
M_1^+= \bigcup_{a \in A, b \in B} (a + b - T(a,b) ), \qquad M_1 = M_1^- \cup M_1^+, $$
and let $f=\one_A, g=\one_B, h= k = \one_{M_1}$. Applying Theorem \ref{thm:main} with the above operation $T$ we  obtain
$$ \sqrt{\# A \, \cdot \#B} \leq \sqrt{\# M_1^- \cdot \# M_1^+} \leq \# M_1. $$ This inequality implies (\ref{eq_257}) since $M_1 \subset M$. Our inequality is quite flexible, as there is nothing canonical about the specific definition \eqref{eq:Ollivier-Villani}
of the map $T$, and moreover the analysis applies for subsets of $\ZZ^n$ and not only for subsets of $\{ 0, 1 \}^n$.

\subsection{\label{sec:Yepes}A discrete Brunn-Minkowski inequality II}
Recently, the following inequality was proven by Iglesias, Yepes Nicolás
and Zvavitch \cite{INZ}:
\begin{thm}\label{thm:Yepes}
	\label{thm:Yepes_BM}Let $\lambda\in\left[0,1\right]$. For any two
	bounded non-empty sets $K,L\sub\RR^{n}$, we have
	\[
	G_{n}\left(\lambda K+\left(1-\lambda\right)L+\left(-1,1\right)^{n}\right)^{1/n}\ge\lambda G_{n}\left(K\right)^{1/n}+\left(1-\lambda\right)G_{n}\left(L\right)^{1/n},
	\]
	where $G_{n}\left(M\right)$ denotes the number of lattice points
	in $M\sub\RR^{n}$.
\end{thm}
We recover a multiplicative version of Theorem \ref{thm:Yepes_BM}
for $\lambda=1/2$: Let $f\left(x\right)=\one_{K}\left(x\right)$,
$g\left(x\right)=\one_{L}\left(x\right)$ be the indicator functions
of $K$ and $L$, and $h=\one_{\frac{1}{2}(K+L)+(-1,0]^{n}}$,  $k=\one_{\frac{1}{2}(K+L)+[0,1)^{n}}$.
Note that for every $x\in K$ and $y\in L$, we have
\begin{align*}
	\flr{\frac{x+y}{2}} \in\frac{K+L}{2}-[0,1)^{n}\textit{ and }\ceil{\frac{x+y}{2}}\in\frac{K+L}{2}+[0,1)^{n},
\end{align*}
which implies that $f\left(x\right)g\left(y\right)\le h\left(\flr{\frac{x+y}{2}}\right)k\left(\ceil{\frac{x+y}{2}}\right)$
for all $x,y\in\ZZ^{n}$. By Theorem \ref{thm:lambda_BM}, we have
\[
\sqrt{G_{n}\left(K\right)G_{n}\left(L\right)}\le
\sqrt{ G_{n}\left(\frac{K+L}{2}+[0,1)^{n}\right)
G_{n}\left(\frac{K+L}{2}+(-1,0]^{n}\right)} \le G_{n}\left(\frac{K+L}{2}+(-1,1)^{n}\right),
\]
as follows also from Theorem \ref{thm:Yepes} via the arithmetic/geometric means inequality.

\begin{rem}
We conclude this paper with a little remark on the case of a finite abelian group, where $G = (\ZZ/ p_1 \ZZ) \times \dots \times (\ZZ/p_n \ZZ)$.
There are no additive, complete orderings on such  groups. Therefore, in order to apply Theorem \ref{thm:main}
for four functions $f,g,h,k: G \rightarrow [0, \infty)$, one option  is to define
$$ \tilde{f}(x) = \left \{ \begin{array}{cl} f(\pi(x)) & 0 \leq x_1 < p_1,\ldots, 0 \leq x_n < p_n \\ 0 & \text{otherwise} \end{array} \right. $$
where $\pi: \ZZ^n \rightarrow (\ZZ/ p_1 \ZZ) \times \dots \times (\ZZ/p_n \ZZ) = G$ is the projection map, and similarly to define $\tilde{g}, \tilde{h}, \tilde{k}$.
In the case where the four functions $\tilde{f},\tilde{g}, \tilde{h}, \tilde{k}: \ZZ^n \rightarrow [0, \infty)$ satisfy the assumptions of Theorem  \ref{thm:main},
we obtain the inequality
\[
\Bp{\sum_{x\in G}f\left(x\right)}\Bp{\sum_{x\in G}g\left(x\right)}\le\Bp{\sum_{x\in G}h\left(x\right)}\Bp{\sum_{x\in G}k\left(x\right)}.
\]
\end{rem}

\bibliographystyle{amsplain_Abr_NoDash}
\bibliography{DBM}

\end{document}